\documentclass[12pt,a4paper,english,reqno]{amsart}
\usepackage[a4paper,footskip=1.5em]{geometry}
\usepackage{amsmath,amssymb,amsthm,mathtools}
\usepackage[mathscr]{euscript}
\usepackage[usenames,dvipsnames]{color}
\usepackage{adjustbox,tikz,calc,graphics,babel,standalone}
\usepackage{subcaption}
\usepackage{csquotes,enumerate,verbatim}
\usepackage[final]{microtype}
\usepackage[numbers]{natbib}
\usetikzlibrary{shapes.misc,calc,intersections,patterns,decorations.pathreplacing}
\usepackage{hyperref}
\hypersetup{colorlinks=true,linkcolor=blue,citecolor=blue,pdfpagemode=UseNone,pdfstartview={XYZ null null 1.00}}
\usepackage{cmtiup}

\theoremstyle{plain}
\newtheorem*{theorem*}{Theorem}
\newtheorem{theorem}{Theorem}[section]

\newtheorem{claim}[theorem]{Claim}
\newtheorem{proposition}[theorem]{Proposition}
\newtheorem*{claim*}{Claim}

\theoremstyle{remark}

\DeclareMathOperator\Deg{d}
\newcommand{\ex}{\mathrm{ex}}
\newcommand{\Forb}{\mathrm{Forb}} 

\newcommand{\CC}{\mathscr{C}}
\newcommand{\CH}{\mathcal{H}}
\newcommand{\CF}{\mathscr{F}}

\newcommand{\N}{\mathbb{N}}

\let\eps\varepsilon
\let\originalleft\left
\let\originalright\right
\renewcommand{\left}{\mathopen{}\mathclose\bgroup\originalleft}
\renewcommand{\right}{\aftergroup\egroup\originalright}

\begin{document}

\title{The number of hypergraphs without linear cycles}

\author{J\'{o}zsef Balogh}
\address{Department of Mathematics, University of Illinois, 1409 W.\/ Green Street, Urbana IL 61801, USA}
\email{jobal@math.uiuc.edu}

\author{Bhargav Narayanan}
\address{Department of Pure Mathematics and Mathematical Statistics, University of Cambridge, Wilberforce Road, Cambridge CB3\thinspace0WB, UK}
\email{b.p.narayanan@dpmms.cam.ac.uk}

\author{Jozef Skokan}
\address{Department of Mathematics, London School of Economics, Houghton Street, London WC2A\thinspace2AE, UK, \emph{and} Department of Mathematics, University of Illinois, 1409 W.\/ Green Street, Urbana IL 61801, USA}
\email{j.skokan@lse.ac.uk}

\date{5 June 2017}


\begin{abstract}
The $r$-uniform linear $k$-cycle $C^r_k$ is the $r$-uniform hypergraph on $k(r-1)$ vertices whose edges are sets of $r$ consecutive vertices in a cyclic ordering of the vertex set chosen in such a way that every pair of consecutive edges share exactly one vertex. Here, we prove a balanced supersaturation result for linear cycles which we then use in conjunction with the method of hypergraph containers to show that for any fixed pair of integers $r, k \ge 3$, the number of $C^r_k$-free $r$-uniform hypergraphs on $n$ vertices is $2^{\Theta(n^{r-1})}$, thereby settling a conjecture due to Mubayi and Wang.
\end{abstract}
\maketitle

\section{Introduction}
The general problem of asymptotically enumerating discrete structures with various forbidden substructures has a very rich history. The simplest such question that one may ask is as follows: given a fixed graph $H$, how many $n$-vertex graphs are there that contain no copy of $H$? In the case where the fixed forbidden graph $H$ is not bipartite, reasonably precise estimates are available from the work of Erd\H{o}s, Frankl and R\"odl~\citep{efr}; see also~\citep{ekr}. On the other hand, the case where the fixed forbidden graph $H$ is bipartite remains an active area of investigation since the corresponding `Tur\'an problem' of determining the maximum number of edges in an $n$-vertex $H$-free graph remains open in general; consequently, much of the work in this case has been focused on understanding the behaviour of important prototypical examples such as complete bipartite graphs (see \citep{bip1, bip2}) and even cycles (see \citep{ bip3, bip4}). Of course, one may ask similar questions for various other discrete structures; see~\citep{survey} for a broad overview of the area.

In this paper, we shall be concerned with enumerating uniform hypergraphs. For an integer $r \ge 2$, an \emph{$r$-uniform hypergraph} (or \emph{$r$-graph} for short) is a pair $(V,E)$ of finite sets, where the edge set $E$ is a family of $r$-element subsets of the vertex set $V$. For a fixed $r$-graph $H$ and a natural number $n \in \N$, the corresponding \emph{Tur\'an problem} asks for the determination of $\ex_r (n, H)$, the maximum number of edges in an $n$-vertex $r$-graph that contains no copy of $H$ as a subgraph, and the associated \emph{enumeration problem} asks for the determination of $|\Forb_r(n ,H)|$, where $\Forb_r(n,H)$ denotes the family of all $H$-free $r$-graphs on the vertex set $[n] = \{1, 2, \dots, n\}$. The Tur\'an problem and the enumeration problem for a given $r$-graph are closely related; indeed, for a fixed $r$-graph $H$, we trivially have
\begin{equation}\label{triv}
2^{\ex_r (n,H)} \le |\Forb_r (n ,H)| \le \sum_{i \le \ex_r(n,H)} \binom{\binom{n}{r}}{i}  = n^{O(\ex_r(n,H))}.
\end{equation}
We remark that it is generally believed that the lower bound in~\eqref{triv} is closer to the truth (provided $H$ is not `degenerate'), and indeed, all existing results in the area support this belief.

Mirroring the situation described earlier for graphs, for each $r \ge 3$, the enumeration problem for a fixed forbidden $r$-graph $H$ has a reasonably satisfactory solution in the case where $H$ is not $r$-partite. Indeed, it follows from the work of Nagle, R\"odl and Schacht~\citep{regl} on hypergraph regularity that for any fixed $r$-graph $H$, we have
\begin{equation}\label{regularity}
|\Forb_r (n ,H)| \le 2^{\ex_r (n,H) + o(n^r)};
\end{equation}
since $\ex_r (n,H) = \Theta(n^r)$ for any $r$-graph $H$ that is not $r$-partite, the above bound complements the trivial lower bound in~\eqref{triv}. 

However, the enumeration problem for a fixed forbidden $r$-partite $r$-graph $H$ is poorly understood; indeed, for any such $H$, we know that $\ex_r (n,H) = O(n^{r-\eps})$ for some constant $\eps>0$ depending on $H$ alone, so the upper bound from~\eqref{regularity} is some ways off from the trivial lower bound in~\eqref{triv}. Consequently, as in the case of graphs, it is important to understand the behaviour of prototypical examples of $r$-partite $r$-graphs. Here, following Mubayi and Wang~\citep{MW}, we shall investigate the enumeration problem for one such prototypical family of $r$-partite $r$-graphs, namely, the family of $r$-uniform \emph{linear (or loose) cycles}. 

For integers $r \ge 2$ and $k \ge 3$, the \emph{$r$-uniform linear $k$-cycle $C^r_k$} is the $r$-graph on $k(r-1)$ vertices whose vertices can be ordered cyclically in such a way that the edges are sets of $r$ consecutive vertices in this ordering such that every two consecutive edges share exactly one vertex. It is known from the work of F\"uredi and Jiang~\citep{FJ} and of Kostochka, Mubayi and Verstra\"ete~\citep{KMV} that for any fixed $r, k \ge 3$, we have $\ex_r(n,C^r_k) = \Theta(n^{r-1})$; it then follows from~\eqref{triv} that we trivially have
\begin{equation}\label{cyc-triv}
|\Forb_r (n ,C^r_k)| = 2^{\Omega(n^{r-1})} \hspace{10pt} \text{and} \hspace{10pt} |\Forb_r (n ,C^r_k)| = n^{O(n^{r-1})}
\end{equation}
for any fixed $r, k \ge 3$. Since the lower bound in~\eqref{triv} is generally believed to be closer to the truth, Mubayi and Wang~\citep{MW} made the natural conjecture that 
\begin{equation}\label{MW-conj}
	|\Forb_r (n ,C^r_k)| = 2^{O(n^{r-1})}
\end{equation}
for any fixed $r,k \ge 3$, and also established this conjecture in the case where $r=3$ and $k \ge 3$ is even; they also showed that the trivial upper bound in~\eqref{cyc-triv} is not sharp for general $r$ and $k$, and their improvements over the trivial upper bound were subsequently refined by Han and Kohayakawa~\citep{HK}. Here, we shall establish~\eqref{MW-conj} in full generality, thereby resolving the conjecture of Mubayi and Wang~\citep{MW}; our main result is as follows.

\begin{theorem}\label{enumthm}
For every pair of integers $r, k \ge 3$, there exists $C = C(r,k) > 0$ such that 
\[|\Forb_r (n ,C^r_k)| \le 2^{Cn^{r-1}}\]
for all $n \in \N$.
\end{theorem}

Our proof of Theorem~\ref{enumthm} proceeds by iteratively applying a construction of `hypergraph containers' introduced independently by Balogh, Morris and Samotij~\citep{BMS} and by Saxton and Thomason~\citep{ST}. To make use of the framework of hypergraph containers, we prove a `balanced supersaturation' theorem for linear cycles: this result roughly states that an $r$-graph $G$ on $n$ vertices with significantly more than $\ex_r(n, C^r_k)$ edges contains many copies of $C^r_k$ which are additionally distributed relatively uniformly over the edges of $G$; as remarked in~\citep{MW}, this result might be of independent interest. 

A brief word on how our approach to proving Theorem~\ref{enumthm} compares to the approach adopted by Saxton and Morris~\citep{bip4} to count graphs without even cycles is perhaps in order. At a very high level, both proofs are based on combining an appropriate balanced supersaturation result with the method of hypergraph containers; however, the arguments used to establish balanced supersaturation are qualitatively very different in the two cases. Indeed, as is evident from~\cite{BL}, the problem of enumerating graphs without copies of even cycles more resembles the problem of enumerating $C^r_k$-free \emph{linear} hypergraphs than the problem at hand here.

This paper is organised as follows. We set up some notation and collect together the results we need for the proof of our main result in Section~\ref{s:prelim}. We state and prove our `balanced supersaturation' theorem for linear cycles in Section~\ref{s:supersat}, and then demonstrate how to deduce Theorem~\ref{enumthm} from this result in Section~\ref{s:proof}. We conclude with some discussion of open problems in Section~\ref{s:conc}.

\section{Preliminaries}\label{s:prelim}
For $n \in \N$, we denote the set $\{1, 2, \dots, n\}$ by $[n]$. For a set $X$, we write $\mathcal{P}(X)$ for the set of all subsets of $X$, and given $r \in \N$, we write $X^{(r)}$ for the family of $r$-element subsets of $X$. In this language, an $r$-graph is a pair $(V,E)$ of finite sets with $E \subset V^{(r)}$; also, as is customary, we shall always refer to $2$-graphs as graphs.

Let $G = (V,E)$ be an $r$-graph. For a set $\sigma \subset V(G)$ with $1 \le |\sigma| \le r$, we define $N_G(\sigma)$ to be the set of edges of $G$ containing $\sigma$, i.e.,
\[N_G(\sigma)=\{e \in E(G): \sigma \subset e\},\]
and we define its \emph{degree $\Deg_G(\sigma)$} by $\Deg_G (\sigma)  = |N_G(\sigma)|$. Next, for $1 \le j \le r$, we define the \emph{maximum $j$-degree $\Delta_j(G)$} of $G$ by 
\[\Delta_j(G)=\max\{\Deg_G(\sigma): \sigma\subset V(G) \text{ and } |\sigma|=j\}.\]
Also, we define the \emph{average degree $d(G)$} of $G$ by $ d(G) = r|E(G)|/|V(G)|$. Finally, for $p\in (0,1)$, the \emph{co-degree function $\delta(G,p)$} of $G$ is given by
\[\delta(G,p)=\frac{1}{d(G)} \sum_{j=2}^r \frac{\Delta_j(G)}{p^{j-1}}.\]

Given an $r$-graph $G$, we write $G[U]$ for the subgraph of $G$ induced by a set $U \subset V(G)$. A subset $I \subset V(G)$ of the vertex set of an $r$-graph $G$ is said to be \emph{independent} in $G$ if no edge of $G$ is contained in $I$; equivalently, $I$ is independent in $G$ if $G[I]$ is empty. We shall make use of the following hypergraph container theorem; see~\citep{ST}, for example.

\begin{theorem}\label{thm-container}
For each $r \in \N$, there exist positive constants $c_1 = c_1(r)$, $c_2 = c_2(r)$ and $c_3 = c_3(r)$ such that the following holds for all $N \in \N$. For each $0 < \eps < 1/2$ and each $N$-vertex $r$-graph $G$, if $0 < p < c_1$ is such that $\delta(G,p)\le c_2 \eps$, then there exists a family $\CC \subset \mathcal{P}(V(G))$ of at most
\[ \exp \left(c_3 N p \log(1/p) \log(1/\eps)\right) \]
subsets of $V(G)$ (called the containers of $G$) such that
\begin{enumerate}
\item for each independent set $I \subset V(G)$, there exists a container $U \in \CC$ such that $I \subset U$, and
\item $|E(G[U])|\le\eps |E(G)|$ for each container $U\in\CC$. \qed
\end{enumerate}
\end{theorem}

Next, we introduce some notation for working with linear cycles. For integers $r \ge 2$ and $k \ge 3$, recall that the $r$-uniform linear $k$-cycle $C^r_k$ is the $r$-graph on $k(r-1)$ vertices whose vertices can be ordered cyclically in such a way that the edges (of which there are precisely $k$) are sets of $r$ consecutive vertices in this ordering such that every two consecutive edges share exactly one vertex. The \emph{core} of an $r$-uniform linear $k$-cycle is the set of those $k$ vertices of the cycle which are each contained in two edges of the cycle. For $r \ge 3$, a \emph{skeleton} of an $r$-uniform linear $k$-cycle is a triple $(\mathbf{e}, \mathbf{v}, \mathbf{u})$, where $\mathbf{e} = (e_1, e_2, \dots, e_k)$ is an ordering of the edges of the cycle, $\mathbf{v} = (v_1, v_2, \dots, v_k)$ is an ordering of the core of the cycle, and  $\mathbf{u} = (u_1, u_2, \dots, u_k)$ is a choice of $k$ non-core vertices of the cycle such that
\begin{enumerate}
\item $v_i, v_{i+1} \in e_i$ for each $1 \le i \le k$ with the convention that $v_{k+1} = v_1$, and
\item $u_i \in e_i$ for each $1 \le i \le k$.
\end{enumerate}
Given a skeleton $(\mathbf{e}, \mathbf{v}, \mathbf{u})$ as above, we call the set $\{v_1, u_1, v_2, u_2, \dots, v_k, u_k\}$ the \emph{support} of the skeleton, and we call the graph on the support depicted in Figure~\ref{can-tri} the \emph{canonical triangulation} of the (support of the) skeleton.

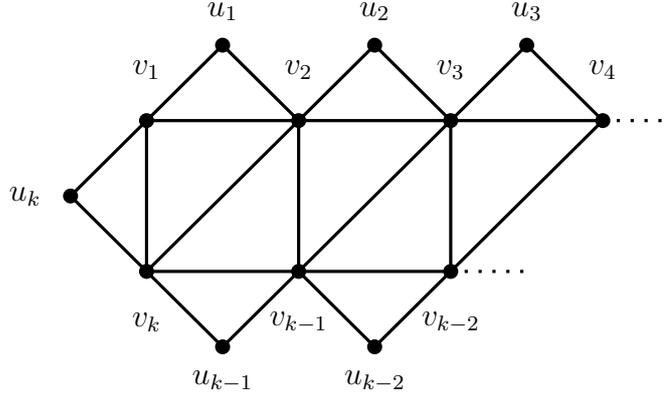
\begin{figure}
	\begin{center}
		
		\begin{tikzpicture}
		
		\foreach \x in {1,3,5}
		\node (\x) at (\x, 0) [inner sep=0.7mm, circle, fill=black!100] {};
		\foreach \x in {1,3,5,7}
		\node (\x) at (\x, 2) [inner sep=0.7mm, circle, fill=black!100] {};
		\foreach \x in {1,3,5}
		\draw [very thick] (\x, 0) -- (\x, 2);
		\foreach \x in {1,3,5}
		\draw [very thick] (\x, 0) -- (\x+2, 2);
		\draw [very thick] (1,0) -- (5,0);
		\draw [very thick] (1,2) -- (7,2);
		
		\draw [loosely dotted, very thick] (7,2) -- (8,2);
		\draw [loosely dotted, very thick] (5,0) -- (6,0);	
		
		\node at (0, 1) [inner sep=0.7mm, circle, fill=black!100] {};
		\foreach \x in {1,3,5}
		\node (\x) at (\x + 1, 3) [inner sep=0.7mm, circle, fill=black!100] {};
		\foreach \x in {1,3,5}
		\draw [very thick] (\x, 2) -- (\x+1, 3) -- (\x +2, 2);
		\foreach \x in {1,3}
		\node (\x) at (\x + 1, -1) [inner sep=0.7mm, circle, fill=black!100] {};
		\foreach \x in {1,3}
		\draw [very thick] (\x, 0) -- (\x+1, -1) -- (\x +2, 0);
		
		\draw [very thick] (1,0) -- (0,1) -- (1,2);
		\node at (-0.6, 1) [] {$u_{k}$};
		\node at (1, -0.675) [] {$v_{k}$};
		\node at (3, -0.675) [] {$v_{k-1}$};
		\node at (5, -0.675) [] {$v_{k-2}$};
		\node at (2, -1.45) [] {$u_{k-1}$};
		\node at (4, -1.45) [] {$u_{k-2}$};
		\node at (1, 2.675) [] {$v_{1}$};
		\node at (3, 2.675) [] {$v_{2}$};
		\node at (5, 2.675) [] {$v_{3}$};
		\node at (7, 2.675) [] {$v_{4}$};
		\node at (2, 3.45) [] {$u_{1}$};
		\node at (4, 3.45) [] {$u_{2}$};
		\node at (6, 3.45) [] {$u_{3}$};
		\end{tikzpicture}
	\end{center}
	\caption{The canonical triangulation of a skeleton.}\label{can-tri}
\end{figure}

We shall use skeletons and their canonical triangulations to help simplify the bookkeeping when counting the number of copies of $C^r_k$ in an $r$-graph. We begin with the following simple observation.
\begin{proposition}\label{overcount}
An $r$-uniform linear $k$-cycle possesses exactly $2k(r-2)^k$ different skeletons. \qed
\end{proposition}

Next, we observe the following property of canonical triangulations.
\begin{proposition}\label{triorder}
For any proper subset $U$ of the support of a skeleton inducing at least one triangle in the corresponding canonical triangulation, there exists a vertex $x \notin U$ in the support with the property that there exist $x',y, z \in U$ such that the sets $\{x,y,z\}$ and $\{x',y,z\}$ both induce triangles in the canonical triangulation.
\end{proposition}
\begin{proof}
The claim follows from the simple observation that for each triangle $\{p,q,r\}$ of the canonical triangulation, there exists an ordering $(x_1, x_2, \dots, x_{2k-3})$ of the remaining $2k-3$ vertices of the support such that for each $i \in [2k-3]$, there exist $x'_i, y_i, z_i \in \{p, q, r, x_1, x_2, \dots, x_{i-1}\}$ such that the sets $\{x_i ,y_i ,z_i\}$ and $\{x'_i ,y_i ,z_i\}$ both induce triangles in the canonical triangulation.
\end{proof}
Finally, to prove Theorem~\ref{enumthm}, we shall need the existence of graphs satisfying a very mild regularity condition; in particular, we shall make use of the following simple fact.

\begin{proposition}\label{almost-reg}
For each $t \in \N$ and each $n \ge t$, there exists an $n$-vertex graph with the property that each vertex of the graph has degree either $t$ or $t-1$. \qed
\end{proposition}

A few more remarks about notation are in order before we proceed. We shall make use of standard asymptotic notation; in the sequel, constants suppressed by the asymptotic notation are allowed to depend on the fixed parameters $r$ and $k$. For the sake of clarity of presentation, we systematically omit floor and ceiling signs whenever they are not crucial.

\section{Balanced supersaturation}\label{s:supersat}
The purpose of this section is to prove the following balanced supersaturation result for linear cycles which asserts that in any $r$-graph on $[n]$ with significantly more than $\ex_r (n, C^r_k)$ edges, one can find many copies of $C^r_k$ that are additionally `well-distributed' across the $r$-graph in question.

\begin{theorem}\label{t:supersat}
For every pair of integers $r,k \ge 3$, there exists $K = K(r,k) >0$ such that the following holds for all $n \in \N$. Given an $r$-graph $G$ on $[n]$ with $|E(G)| = tn^{r-1}$ for some $t \ge 2k(r-1) $, there exists a $k$-graph $\CH$ on $E(G)$, where each edge of $\CH$ is a copy of $C^r_k$ in $G$, such that
\begin{enumerate}
\item $d(\CH) \ge K^{-1} t^{k-2}D^{k-1}$, and
\item $\Delta_j (\CH) \le K t^{k-j-1} D^{k-j}$ for each $1 \le j \le k-1$,
\end{enumerate}
where $D = t^2 n^{r-4}$ if $r \ge 4$ and $D = t$ if $r=3$.
\end{theorem}

\begin{proof}
We start by setting $D_2 = tn^{r-3}$ and $D_3 = D/t$; note that since $D = t^2 n^{r-4}$ if $r \ge 4$ and $D = t$ if $r=3$, we have $D_3 = tn^{r-4}$ if $r \ge 4$ and $D_3 = 1$ if $r = 3$.

We shall construct three hypergraphs on $[n]$ from $G$: an $r$-graph $F$ which is a subgraph of $G$, a $3$-graph $A$ and a graph $B$. First, we obtain $F$ from $G$ by repeatedly applying, until it is no longer possible to do so, the following deletion rule: if there exists either a $3$-set $\sigma \subset [n]^{(3)}$ with $0 < \Deg_G(\sigma) < D_3$ or a $2$-set $\pi \subset [n]^{(2)}$ with $0 < \Deg_G(\pi) < D_2$, delete every edge of $G$ containing the corresponding set. We then define $A$ to be the $3$-graph whose edge set consists of those triples $\sigma \in [n]^{(3)}$ such that $\Deg_F(\sigma) > 0$, and analogously define $B$ to be the graph whose edge set consists of those pairs $\pi \in [n]^{(2)}$ such that $\Deg_F(\pi) > 0$. 

We observe that $F$, $A$ and $B$ have the following properties by virtue of how they are constructed.
\begin{enumerate}[(A)]
\item\label{F-deg} First, $\Deg_F(\sigma) \ge D_3$ for each $\sigma \in E(A)$ and $\Deg_F(\sigma) = 0$ for all $\sigma \in [n]^{(3)} \setminus E(A)$, and analogously, $\Deg_F(\pi) \ge D_2$ for each $\pi \in E(B)$ and  $\Deg_F(\pi) = 0$ for all $\pi \in [n]^{(2)} \setminus E(B)$
\item\label{BtoA} Next, we have $\Deg_A(\pi) \ge t$ for each $\pi \in E(B)$. Indeed, if this fails to hold for some $\pi \in E(B)$, then \[\Deg_F(\pi) \le \Deg_A(\pi) n^{r-3} < t n^{r-3} = D_2 ,\] a contradiction.
\item\label{eF} Finally, we have $|E(F)| \ge |E(G)|/3$. Indeed, the number of edges deleted from $G$ to obtain $F$ is at most 
\[ \binom{n}{2} D_2 \le tn^{r-1}/2 \]
in the case where $r=3$, and at most
\[\binom{n}{2}D_2 + \binom{n}{3}D_3 \le tn^{r-1}/2 + tn^{r-1}/6 = 2tn^{r-1}/3.\]
in the case where $r \ge 4$.

\end{enumerate}

We shall construct $\CH$ from the copies of $C^r_k$ in $F$ using $A$ and $B$; to do so it will be helpful to define some auxiliary structures. First, we define a collection of graphs on $E(A)$, one for each $\pi \in E(B)$, as follows. From~\eqref{BtoA}, we know that $\Deg_A(\pi) \ge t$ for each $\pi \in E(B)$, so we may appeal to Proposition~\ref{almost-reg} and fix a graph $\Gamma(\pi)$ on $N_A (\pi)$ with the property that each $\sigma \in N_A(\pi)$ has degree either $t$ or $t-1$ in $\Gamma(\pi)$. Next, let $\Gamma$ be the graph on $E(A)$ whose edge set is the union of the edge sets of the graphs $\Gamma(\pi)$ over $\pi \in E(B)$; since $\sigma \in E(A)$ has positive degree in $\Gamma(\pi)$ if and only if $\pi \subset \sigma$, it follows that the degree of each $\sigma \in E(A)$ in $\Gamma$ is at most $3t$. Finally, we know from~\eqref{F-deg} that $\Deg_F(\sigma) \ge D_3$ for each $\sigma \in E(A)$; we may therefore fix a subset $\Lambda(\sigma) \subset N_F(\sigma)$ of size $D_3$ for each $\sigma \in E(A)$.

We shall construct $\CH$ by specifying a skeleton (though possibly more than one) for each copy of $C^r_k$ in $F$ that we wish to include in $\CH$. Furthermore, we shall guarantee that the support of each skeleton that we specify has the following \emph{adjacency property}: if $\{x,y,z_1\}$ and $\{x,y,z_2\}$ are two subsets of the support that induce triangles in the corresponding canonical triangulation, then $\{x,y,z_1\}$ and $\{x, y, z_2\}$ are both edges of $A$ that are adjacent in the graph $\Gamma(\{x,y\})$. 

We now describe an algorithm to construct skeletons of copies of $C^r_k$ in $F$ whose supports additionally satisfy the adjacency property; recall that specifying a skeleton of a copy of $C^r_k$ in $F$ involves specifying a triple $(\mathbf{e}, \mathbf{v}, \mathbf{u})$, where $\mathbf{e} = (e_1, e_2, \dots, e_k)$ is an ordering of the edges of the cycle, $\mathbf{v} = (v_1, v_2, \dots, v_k)$ is an ordering of the core of the cycle, and  $\mathbf{u} = (u_1, u_2, \dots, u_k)$ is a choice of $k$ non-core vertices of the cycle.
\begin{enumerate}[(i)]
\item\label{st1} We start by choosing an edge $e$ of $F$, in $|E(F)|$ ways, and setting $e_k = e$. We then choose three distinct vertices from $e$ and designate these vertices to be $v_1$, $v_k$ and $u_k$ in some order.
\item\label{st2} Next, we specify $v_2, v_3, \dots, v_{k-1}$ inductively as follows. Having specified $v_1, v_2, \dots, v_i$ and $v_k, v_{k-1}, \dots v_{k-i+1}, u_{k-i+1}$, we first fix $v_{i+1}$ as follows. Consider the triple $\sigma = \{v_{i}, v_{k-i+1}, v\}$, where $v = u_k$ if $i =1$ and $v = v_{k-i + 2}$ if $i \ge 2$, and the pair $\pi = \{v_{i}, v_{k-i+1}\}$. Let $\sigma' = \{v_{i}, v_{k-i+1}, v'\}$ be a neighbour of $\sigma$ in the graph $\Gamma(\pi)$ chosen in such a way that $v'$ is distinct from all the already specified vertices of the support and so that $v' \notin e_k$. We then set $v_{i+1} = v'$, noting that since $\sigma$ has at least $t-1$ neighbours in $\Gamma(\pi)$, there are at least $t - 1 -k(r-1) \ge t/3$ choices for $v_{i+1}$. If $k$ is odd and it so happens that $k-i = i+1$, then we stop after fixing $v_{i+1}$. If not, then we pick $v_{k-i}$ in a manner analogous to how we chose $v_{i+1}$, working instead with the triple $\sigma = \{v_{i+1}, v_{k-i + 1}, v_i\}$ and the pair $\pi = \{v_{i+1}, v_{k-i + 1}\}$. If $k$ is even and it so happens that $k - i = i + 2$, then we stop after fixing $v_{k-i}$. Observe that these choices ensure that the subset of the support specified so far satisfies the adjacency property.
\item\label{st3} Now, we finish specifying the support of the skeleton by inductively choosing $u_1, u_2, \dots, u_{k-1}$ as follows. For $i \in [k-1]$, having specified, $u_1, u_2, \dots, u_{i-1}$, we fix $u_{i}$ as follows. Note that there exists a unique vertex $v \in \{v_1, v_2, \dots, v_k\}$ such that the triple $\{v_i, v_{i+1}, v\}$ induces a triangle in the canonical triangulation; let $\sigma = \{v_i, v_{i+1}, v\}$ and $\pi = \{v_i, v_{i+1}\}$. Let $\sigma' = \{v_{i}, v_{i+1}, u\}$ be a neighbour of $\sigma$ in the graph $\Gamma(\pi)$ chosen in such a way that $u$ is distinct from all the already specified vertices of the support and so that $u \notin e_k$. We then set $u_i = u$, noting that since $\sigma$ has at least $t-1$ neighbours in $\Gamma(\pi)$, there are again at least $t - 1 -k(r-1) \ge t/3$ choices for $u_i$. Again, note that these choices ensure that the support satisfies the adjacency property.
\item\label{st4} Finally, we finish specifying the skeleton by fixing $e_1, e_2, \dots, e_{k-1}$ inductively as follows. For $i \in [k-1]$, having specified, $e_1, e_2, \dots, e_{i-1}$, we fix $e_{i}$ as follows. Let $\sigma = \{v_i, v_{i+1}, u_i\}$ and consider the set $\Lambda(\sigma) \subset N_F(\sigma)$ of $D_3$ edges in $F$ containing $\sigma$. Choose an edge $e \in \Lambda(\sigma)$ with the property that $e \setminus \sigma$ is disjoint both from $e_k \cup e_1 \cup e_2 \cup \dots \cup e_{i-1}$ and from the support of the skeleton. We then set $e_i = e$, noting that the number of choices for $e_i$ is at least
\[ D_3 - k(r-1) n^{r-4} = tn^{r-4} - k(r-1) n^{r-4} \ge tn^{r-4}/2 = D_3/2\]
in the case where $r\ge 4$, and trivially at least $1 \ge D_3/2$ in the case where $r = 3$ (since specifying the support specifies the entire skeleton when $r=3$).
\end{enumerate}

We define $\CH$ by setting $V(\CH) = E(G)$ and including a copy of $C^r_k$ in $F$ in $E(\CH)$ if and only if at least one skeleton of the cycle in question is generated by the above algorithm. 

We now show that $\CH$ satisfies the requisite degree conditions. We remind the reader that constants suppressed by the asymptotic notation in what follows are allowed to depend on the fixed parameters $r$ and $k$. We start by bounding the average degree of $\CH$ from below.

\begin{claim}\label{cl1}$d(\CH) = \Omega( t^{k-2}D^{k-1})$.
\end{claim}
\begin{proof}
By Proposition~\ref{overcount}, a fixed copy of $C^r_k$ in $F$ possesses $2k(r-2)^k = O(1)$ distinct skeletons; therefore, it suffices to show that the number of distinct skeletons generated by the algorithm described above is $\Omega(|E(G)| t^{k-2}D^{k-1})$. To do this, we count the number of distinct choices available to us at each stage in the above algorithm. Indeed, in the first stage~\eqref{st1}, we have at least $|E(F)|$ distinct choices, in the second stage~\eqref{st2}, we have at least $(t/3)^{k-2}$ distinct choices, in the third stage~\eqref{st3}, we have at least $(t/3)^{k-1}$ distinct choices, and in the final stage~\eqref{st4}, we have at least $(D_3/2)^{k-1}$ distinct choices. We know from~\eqref{eF} that $|E(F)| \ge |E(G)|/3$, so we have
\[ |E(\CH)| = \Omega\left(|E(F)| t^{k-2} (tD_3)^{k-1}\right) = \Omega\left(|E(G)| t^{k-2} D^{k-1}\right);\]
the claim follows since
\[d(\CH) = k|E(\CH)|/|E(G)| = \Omega\left( t^{k-2}D^{k-1}\right). \qedhere\]

\end{proof}

To finish the proof, we bound the maximum degrees of $\CH$ from above.

\begin{claim}\label{cl2}  For each $j \in[k-1]$, we have $\Delta_j (\CH) = O(t^{k-j-1} D^{k-j})$.
\end{claim}
\begin{proof}
Fix $j \in [k-1]$ and a set $S \subset E(G)$ of size $j$. Our aim is to show that $\Deg_\CH (S) = O(t^{k-j-1} D^{k-j})$; to do this, we shall bound from above the number of distinct skeletons $(\mathbf{e}, \mathbf{v}, \mathbf{u})$ generated by our algorithm that contain $S$, i.e., with the property that $S \subset \{e_1, e_2, \dots, e_k\}$. Observe that we may assume that $S \subset E(F)$, for if not, the number of such skeletons, and consequently $\Deg_\CH(S)$, is zero.

First, we fix which edges of the skeleton correspond to which edges in $S$; this may be done in at most $k^j = O(1)$ ways. Next, for each edge in $S$, we identify the two core vertices and one non-core vertex from that edge that belong to the support of the skeleton and note that this may be done in at most $jk^3 = O(1)$ ways; let $X$ denote the subset of the support contained in some edge in $S$ and note that since two edges of an $r$-uniform linear $k$-cycle intersect in at most one core vertex, and since each core vertex belongs to precisely two edges, we must have $|X| \ge 2j+1$.

We first estimate the number of ways of choosing, from $[n]$, the vertices in the support of the skeleton not in $X$. By repeatedly applying Proposition~\ref{triorder}, we know that it is possible to find an ordering $(x_1, x_2, \dots, x_m)$ of the $m \le 2k - (2j+1)$ vertices of the support not in $X$ with the property that each for each $x_i$, there exist $x'_i, y_i, z_i \in X \cup \{x_1, x_2, \dots x_{i-1}\}$ such that the sets $\{x_i ,y_i ,z_i\}$ and $\{x'_i ,y_i ,z_i\}$ both induce triangles in the canonical triangulation. We now count the number of ways to choose, for $1 \le i \le m$, the vertex $x_i$ from $[n]$. Having picked $x_1, x_2, \dots, x_{i-1}$, we know from the previous observation, and from the adjacency property of the support guaranteed by our algorithm, that there exist two triples $\sigma, \sigma' \in E(A)$ that are adjacent in $\Gamma$ such that $x_i \in \sigma$ and $\sigma' \subset X \cup \{x_1, x_2, \dots x_{i-1}\}$. Consequently, the number of choices for $x_i$ is at most the number of ways of choosing $\sigma'$ multiplied by the maximum degree of $\Gamma$, which is at most $(2k)^3 3t = O(t)$. It follows that the number of ways of choosing the support of a skeleton containing $S$ is $O(t^{2k - (2j+1)})$. 

Next, having fixed the support of a skeleton containing $S$, we need to choose the $k-j$ remaining edges of the skeleton. The number of ways of doing this is  easily seen to be at most $D_3 ^{k-j}$. 

Finally, combining the above estimates, we see that
\[\Deg_\CH (S) = O\left(t^{2k - (2j+1)}D_3 ^{k-j} \right) = O\left(t^{k-j-1} D^{k-j}\right );\]
this establishes the claim.
\end{proof}

The result now follows from Claims~\ref{cl1} and~\ref{cl2} by choosing $K = K(r,k)$ to be suitably large.
\end{proof}

\section{Proof of the main result}\label{s:proof}
We now show how to deduce Theorem~\ref{enumthm} from Theorem~\ref{t:supersat} using Theorem~\ref{thm-container}. We shall establish our main result through iterated applications of the following proposition.

\begin{proposition}\label{itercont}
For every pair of integers $r,k \ge 3$, there exists $L = L(r,k) >0$ such that the following holds for all $n \in \N$. Given an $r$-graph $G$ on $[n]$ with $|E(G)| = tn^{r-1}$ for some $t \ge L$, there exists a collection $\CC(G)$ of at most
\begin{equation}\label{contbound} 
\exp\left(\frac{L n^{r-1}}{t^{1/4}}\right)
\end{equation}
subgraphs of $G$ such that
\begin{enumerate}
\item each $C^r_k$-free subgraph $F$ of $G$ is a subgraph of some $H \in \CC(G)$, and
\item $|E(H)| \le (1 - 1/L) |E(G)|$ for each $H \in \CC(G)$.
\end{enumerate}
\end{proposition}
\begin{proof}
Let $K = K(r,k)$ be as promised by Theorem~\ref{t:supersat}, and let $c_1 = c_1 (k)$, $c_2 = c_2(k)$ and $c_3 = c_3(k)$ be as promised by Theorem~\ref{thm-container}. Also, as in Theorem~\ref{t:supersat}, we set $D = t^2 n^{r-4}$ if $r \ge 4$ and $D = t$ if $r=3$, noting that this ensures that $D \ge t$ for all $r \ge 3$.

Now, fix $\eps = 1/4$ and $p = LD^{-1}t^{-(k-2)/(k-1)}$, where we define $L$, with the benefit of hindsight, by
\[ L = L(r,k) = 10 + \max \left\{2k(r-1), \frac{4(K + kK^2)}{c_2}, (4c_3)^{12}, 2kK^2\right\}.
\]

Since $t \ge L \ge 2k(r-1)$, we may apply Theorem~\ref{t:supersat} to $G$ to get a $k$-graph $\CH$ on $E(G)$, where each edge of $\CH$ is a copy of $C^r_k$ in $G$, such that
\begin{enumerate}
\item $d(\CH) \ge K^{-1} t^{k-2}D^{k-1}$, and
\item $\Delta_j (\CH) \le K t^{k-j-1} D^{k-j}$ for each $1 \le j \le k-1$.
\end{enumerate}
First, note that if $F$ is a $C^r_k$-free subgraph of $G$, then $F$ is an independent set in $\CH$. Next, from the above bounds and the fact that  $L \ge 4(K+kK^2)/ c_2$, it is easily verified that the co-degree function $\delta(\CH, p)$ of $\CH$ satisfies
\begin{align*}
\delta(\CH, p) &= \frac{1}{d(\CH)} \sum_{j=2}^k \frac{\Delta_j(\CH)}{p^{j-1}} \le \frac{K}{t^{k-2}D^{k-1}} \left(\frac{1}{p^{k-1}}  + \sum_{j=2}^{k-1}\frac{ K t^{k-j-1} D^{k-j}}{p^{j-1} } \right)\\
&= \frac{K}{L^{k-1}} + \sum_{j = 2}^{k-1} \frac{K^2}{L^{j-1}t^{(j-1)/(k-1)}} \le \frac{K}{L^{k-1}} + \frac{kK^2}{L} \le \frac{K + kK^2}{L} \le c_2 \eps.
\end{align*}
We may therefore apply Theorem~\ref{thm-container} to the $k$-graph $\CH$ to get a collection $\CC$ of at most
\[ \exp \left(c_3 |E(G)| p \log(1/p) \log(1/\eps)\right) \]
subgraphs of $G$ such that
\begin{enumerate}
\item each $C^r_k$-free subgraph $F$ of $G$ is a subgraph of some $H \in \CC$, and
\item $|E(\CH[H])| \le \eps |E(\CH)|$ for each $H \in \CC$.
\end{enumerate}

To finish the proof, we proceed as follows. First, since $D\ge t \ge L \ge (4c_3)^{12}$ and $t^{(k-2)/(k-1)} \ge t^{1/3}$ for all $k \ge 3$, it plainly follows that
\begin{align*}
c_3 |E(G)| p \log(1/p) \log(1/\eps) &\le 2c_3 \left(\frac{Ltn^{r-1} \log \left(Dt^{(k-2)/(k-1)}\right)} {Dt^{(k-2)/(k-1)}} \right)\\
&\le 4c_3 \left(\frac{Ln^{r-1} \log t} {t^{(k-2)/(k-1)}} \right)\le \frac{Ln^{r-1}}{t^{1/4}}.
\end{align*}
Therefore, we have $|\CC| \le \exp(Ln^{r-1}/t^{1/4})$. Next, we know that each $H \in \CC$ satisfies $|E(\CH[H])| \le \eps |E(\CH)|$; we claim that this implies that $|E(H)| \le (1-1/L)|E(G)|$. To see this, note that 
\[|E(\CH[H])| \ge |E(\CH)| - \Delta_1(\CH) (|E(G)| - |E(H)|),\] 
so if $|E(H)| > (1- 1/L)|E(G)|$, then since $L \ge 2kK^2$ and $\Delta_1(\CH) \le K^2 d(\CH)$, we have
\[ |E(\CH[H])| > \frac{|E(G)|d(\CH)}{k} - \frac{K^2 |E(G)|d(\CH)}{L} \ge \frac{|E(G)| d(\CH)}{4k} = \eps |E(\CH)|,\]
which is a contradiction. It follows that $\CC = \CC(G)$ is indeed the required collection of subgraphs of $G$.
\end{proof}

We are now in a position to prove Theorem~\ref{enumthm}.
\begin{proof}[Proof of Theorem~\ref{enumthm}]
Let $L = L(r,k)$ be as promised by Proposition~\ref{itercont}. We wish to estimate the number of $C^r_k$-free $r$-graphs on $[n]$; equivalently, we wish to estimate the number of $C^r_k$-free subgraphs of the complete $r$-graph on $[n]$.  To this end, we define a sequence $(t_i)_{i=1}^m$ of positive reals, and a sequence $(\CF)_{i=0}^{m}$ of families of $r$-graphs as follows. We set $t_1 = \binom{n}{r} / n^{r-1}$ and define $t_i = (1-1/L)t_{i-1}$, with $t_m$ being the first term of this sequence to satisfy $t_m \le 2L$. We take $\CF_0$ to consist of a single $r$-graph on $[n]$, namely, the complete $r$-graph on $[n]$, and for $1 \le i \le m$, we obtain $\CF_i$ from $\CF_{i-1}$ by replacing each $r$-graph $G \in \CF_{i-1}$ for which $|E(G)| > Ln^{r-1}$ by the collection $\CC(G)$ of its subgraphs guaranteed by Proposition~\ref{itercont}. 

Now, let $\CF = \CF_{m}$. It is clear that each $C^r_k$-free $r$-graph on $[n]$ is a subgraph of some $G \in \CF$. Furthermore, it is easy to see that $|E(G)| \le Ln^{r-1}$ for each $G \in \CF$. Finally, a simple induction using~\eqref{contbound} shows that
\[ |\CF| \le \exp \left( \sum_{i=1}^{m} \frac{L n^{r-1}}{t_i^{1/4}}\right) \le \exp\left( 10 L^2 n^{r-1}\right), \]
where the last inequality holds on account of the sequence $(t_i)_{i=1}^m$ decreasing geometrically. 

It follows that the number of $C^r_k$-free $r$-graphs on $[n]$ is bounded above by
\[ \sum_{G \in \CF} 2^{|E(G)|} \le |\CF| \exp\left(L n^{r-1}\right) \le  \exp\left(10L^2n^{r-1} + Ln^{r-1}\right);\]
the result follows, with room to spare, by setting $C(r,k) = 20L^2 + 2L$.
\end{proof}
\section{Conclusion}\label{s:conc}
Our results in this paper raise the natural question of deciding, for a fixed pair of integers $r,k \ge 3$, if it is the case that
\[ |\Forb_r (n ,C^r_k)| = 2^{(1+o(1))\ex_r (n,C^r_k)}  \]
as $n\to \infty$. While it is plausible that such an estimate is true, we conclude by warning the reader that an analogous estimate for the $2$-uniform $6$-cycle was shown \emph{not} to hold by Saxton and Morris~\citep{bip4}.

\section*{Acknowledgements}
The first and third authors were partially supported by NSF Grant DMS-1500121 and the first author also wishes to acknowledge support from an Arnold O. Beckman Research Award (UIUC Campus Research Board 15006) and the Langan Scholar Fund (UIUC).

Some of the research in this paper was carried out while the first author was a Visiting Fellow Commoner at Trinity College, Cambridge and the third author was visiting the University of Cambridge; we are grateful for the hospitality of both the College and the University.
\bibliographystyle{amsplain}
\bibliography{linear_cycles}

\end{document}